\documentclass[letterpaper,11pt]{article}

\usepackage{amsmath,amssymb, amsfonts,bm, mathtools}
\usepackage[margin=1in]{geometry}
\usepackage{amsthm}
\usepackage{tikz}

\usepackage{textcomp}
\usepackage{subcaption}
\usepackage{graphicx} 
\usepackage[font=small,labelfont=bf]{caption} 
\usepackage[subrefformat=parens]{subcaption}
\usepackage[usenames,dvipsnames]{pstricks}
\usepackage{epsfig}
\usepackage{floatrow}

\usepackage{pst-grad} 
\usepackage{pst-plot} 
\usepackage{float}
\usepackage[space]{grffile} 
\usepackage{etoolbox} 
\usepackage[pagebackref]{hyperref}  
\usepackage{hyperref}

\newtheorem{theorem}{Theorem}[section]
\newtheorem{lemma}[theorem]{Lemma}
\theoremstyle{definition}

\newtheorem{example}[theorem]{Example}

\newtheorem{proposition}[theorem]{Proposition}
\newtheorem{corollary}[theorem]{Corollary}

\theoremstyle{remark}

\newcommand{\Mod}[1]{\ (\mathrm{mod}\ #1)}
\newfloatcommand{capbtabbox}{table}[][\FBwidth]

\begin{document} 
	\begin{center}
		\textbf{Decomposition of complete tripartite graphs into 5-cycles\\
			using the graph adjacency matrix}
		
			Bahareh Kudarzi$^1$, E. S. Mahmoodian$^1$, Zahra Naghdabadi$^1$\\
			$^1$\textit{Department of Mathematical Sciences, Sharif University of Technology, Tehran, Iran}
	\end{center}		
		
		\begin{abstract}
			The problem of decomposing a complete tripartite graph into $5$-cycles was first proposed in 1995 by Mahmoodian and Mirzakhani and since then many attempts have been made to decompose such graphs into $5$-cycles. Such attempts were partially successful but parts of the problem still remain open. In this paper, we investigate the problem deeper, decompose more tripartite graphs into 5-cycles, and introduce the Graph Adjacency Matrix (GAM) method for cycle decomposition in general. GAM method transforms the cycle decomposition problem to covering squares with certain polygons. This new formulation is easier to solve and enables us to find explicit decompositions for numerous cases that were not decomposed before.\\
			
		\end{abstract}

	\section*{Introduction}
	In mathematics, decomposition of general structures into simple specific units provides important information to the topic. In graph theory, it is always a good question to ask whether a graph $G$ can be decomposed into cycles and paths or not. By a Theorem of Nash-Williams, a graph has a decomposition into cycles if and only if it does not contain an odd cut \cite{nash1960decomposition,thomassen2017nash}. 
	
	It is known that bipartite graphs only have cycles of even length. It was in 1981 that Sotteau proved the complete bipartite graph $K_{m,n}$ can be decomposed into cycles of length $2k$ whenever $2k$ divides the number of edges and both $m$ and $n$ are even \cite{2}. In 1966, Rosa had proved that $K_{n}$ admits a decomposition into 5-cycles if and only if $n$ is odd and the number of edges is a multiple of 5 \cite{13}.
	
	Considering a complete tripartite graph $K_{r,s,t}$, it is easy to check that $K_{r,s,t}$ can be decomposed into $3$-cycles if and only if $r=s=t$. The next interesting case is about decomposition of $K_{r,s,t}$ into $5$-cycles. In 1995, Mahmoodian and Mirzakhani stated this problem and introduced three necessary conditions for a complete tripartite graph $K_{r,s,t}$ to have a decomposition into 5-cycles \cite{3}. The necessary conditions are as follows.
	\begin{itemize}
		\item $5\ | \ rs + rt + st$,
		\item $r\equiv s\equiv t \ \Mod{2}$, 
		\item $t\leq 4rs/(r+s)$.
	\end{itemize}
	In the same paper, sufficiency of these necessary conditions was proved for the special case of $K_{r,r,s}$ except when 5 divides $r$ but not $s$, and conjectured in the general case \cite{3}. Henceforth, here we call the above conditions “\textit{the necessary conditions}” and sufficiency of necessary conditions is regarded as “\textit{Conjecture}”.
	In 2000 \cite{4}, Cavenagh introduced a systematic representation for decomposing tripartite graphs into 5-cycles and proved that $K_{r,r,s}$ can be decomposed into $5$-cycles when $5$ divides $r$ but not $s$. Using the same representation, he proved the correctness of Conjecture for graphs with even parts in 2002 \cite{12}.
	
	In 2011 \cite{6}, the case where $K_{r,s,t}$ has asymptotically similar parts was considered. Correctness of the Conjecture was proved for $K_{r,s,t}$ where $r,s,t$ are odd and $100 \leq r\leq s\leq t \leq \kappa r$, for $\kappa \simeq 1.08$. In 2012 \cite{7}, the coefficient $\kappa$ was increased to $\frac{5}{3}$. 
	In 2019 , $K_{r,s,t}$ where $r,s,t$ are multiples of $5$ was studied and the Conjecture was proved to be true when $t+90 \leq 4rs\ (r+s)$ and $t\neq s+10$ \cite{8}.
	
	Until now the sufficiency of the necessary conditions is proved in all cases except for odd cases with at least one part coprime with $5$ and the case when $ r \leq s \leq t $, and $t \geq \frac{5}{3} r$. Also, in some cases with small parts, existence of a $5$-cycle decomposition is still unknown. Here, we fill some of these gaps and introduce a simple method of finding cycle decompositions.
	
	In the following, we investigate the problem deeper and find relations between triplets $(r,s,t)$ such that the complete tripartite graph $k_{r,s,t}$ satisfies the necessary conditions, stated previously. Then, we reformulate the problem so that it becomes easier to solve. In this regard, the Graph Adjacency Matrix (GAM) method is introduced that transforms the problem of finding \textit{trades} in Latin squares to covering sub-matrices with polygons. Finally, using GAM method and the Mixing Theorem, we decompose more tripartite graphs, including graphs with odd parts, into 5-cycles.
	
	\section{Decomposition of graphs into 5-cycles} 

	There are some relations between triplets satisfying the necessary conditions.
	In tripartite graphs, each $5$-cycle must have at least one vertex in each part. Figure \ref{fig:1} shows the three possible types of 5-cycles.\\
	\begin{figure}[H] \label{fig 1}
		\centering
		\begin{subfigure}{0.25\textwidth}
			\includegraphics[width=\linewidth]{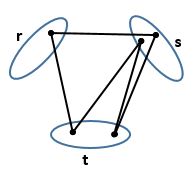}
			\caption{type 1} \label{fig:1a}
		\end{subfigure}
		\begin{subfigure}{0.25\textwidth}
			\includegraphics[width=\linewidth]{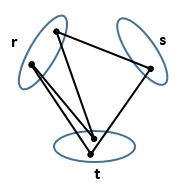}
			\caption{type 2} \label{fig:1b}
		\end{subfigure}
		\begin{subfigure}{0.25\textwidth}
			\includegraphics[width=\linewidth]{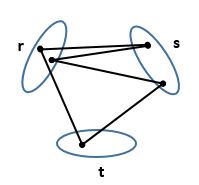}
			\caption{type 3} \label{fig:1c}
		\end{subfigure}
		\caption{The three possible types of 5-cycles are shown in (a), (b), and (c).} \label{fig:1}
	\end{figure}
	Let $c_{1}, c_{2}$, and $c_{3}$ be the number of cycles of type 1,2, and 3, respectively. By counting the number of edges between any two parts we have
	$$3c_{1}+c_{2}+c_{3}= st$$
	$$c_{1}+3c_{2}+c_{3}= rt$$
	$$c_{1}+c_{2}+3c_{3}= rs$$
	Therefore, $c_{1}, c_{2}$, and $c_{3}$ can be computed explicitly as follows.
	\begin{equation}\label{prop 2.2}
	\begin{multlined}
	c_{1}= (4st-rt-rs)/10,\\
	c_{2}=  (-st+4rt-rs)/10, \\
	c_{3}= (-st-rt+4rs)/10. 
	\end{multlined}
	\end{equation}
	
	Also, when we have a decomposition of $K_{r,s,t}$ at hand, by combining $5$-cycles together we can decompose larger graphs. One way of combining graphs gives Proposition \ref{prop2.3}.
	 It is noted that, the convention $r \leq s \leq t$ is not used unless stated directly.
	\begin{proposition} \label{prop2.3} \cite{3}
		Assume that for some $r$, $s$, and $t$ the graph $K_{r,s,t}$ has a decomposition into $5$-cycles. Then $K_{nr, ns, nt}$ also has a decomposition into $5$-cycles for all $n \in \mathbb{N}$.
	\end{proposition}
	
	\begin{proof}[Proof Outline]
		Since $K_{n,n,n}$ can be decomposed into triangles, one can decompose $K_{nr, ns, nt}$ into $n^{2}$ copies of $K_{r, s, t}$. Then by combining $5$-cycles of each $K_{r, s, t}$, we obtain a decomposition for $K_{nr, ns, nt}$.\\
		\begin{figure}[H]
			\centering
			\includegraphics[width=6cm]{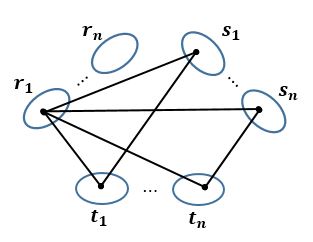}
			\captionof{figure}{Decomposing $K_{nr, ns, nt}$ into copies of $K_{r, s, t}$}
		\end{figure}	
	\end{proof}
	
	Henceforth, we define $K_{r, s, t}$ as a “\textit{primitive}” graph, when it satisfies the necessary conditions in which $(r, s, t)$ is not a multiple of another triplet $(r', s', t')$. For primitive graphs, $ \gcd (r, s, t) =1,2$, or $5$.
	
	\begin{corollary} \cite{4} \label{mod10}
		If the complete tripartite graph $K_{r, s, t}$ can	be decomposed into $5$-cycles, then at least two of the parts are equal modulo $5$.
	\end{corollary}
	\begin{proof}
		Let $r \equiv r' \ \Mod{5}$, $s \equiv s' \ \Mod{5}$ and $t \equiv t' \ \Mod{5}$, where $0 \leq r', s', t' \leq 4$. Then the triplet $(r', s', t')$ must belong to the following list:
		$$(0,0,0),  (0,0,1),  (0,0,2),  (0,0,3),  (0,0,4),
		(0,1,0),  (0,2,0),  (0,3,0),  (0,4,0),  (1,0,0), $$
		$$ (2,0,0),  (3,0,0), (4,0,0),  (1,1,2), (1,2,1),
		(2,1,1),  (1,3,3),  (3,1,3),  (3,3,1),  (2,2,4),$$
		$$(2,4,2),  (4,2,2),  (3,4,4),  (4,3,4),  (4,4,3).$$
		This follows from the fact that $rs + rt + st$ is divisible by $5$; it is straightforward to check that these are all the possible cases.
	\end{proof}
	More specifically, if $r,s$, and $t$ are all odd, then $(r,s,t)$ is equal to one of the following modulo 10 \cite{7}. 
		$$(1,1,7), (1,3,3), (3,9,9), (7,7,9), (1,5,5), (3,5,5), (5,5,5), (7,5,5), (9,5,5)$$
	Moreover in \cite{7}, the complete tripartite graphs $K_{11, 15, 25}$, $K_{13, 15, 25}$, $K_{17, 15, 25}$, $K_{19, 15, 25}$, and $K_{7, 17, 19}$ were decomposed into $5$-cycles. We use these decompositions to decompose larger graphs into 5-cycles.	
	
	\section{On necessary conditions}

	In this section, we investigate the problem deeper and investigate triplets $(r,s,t)$ such that the complete tripartite graph $k_{r,s,t}$ satisfies the necessary conditions, stated previously.
	\begin{lemma} \label{3.14}
		If the complete tripartite graph $K_{r,s,t}$ satisfies the necessary conditions, then\\ $K_{ar+ 5n, as+5n, at+5n}$ also satisfies the necessary conditions for every $a,n\in \mathbb{N}$.
	\end{lemma}
	
	By Lemma \ref{3.14}, from each $(r,s,t)$ satisfying the necessary conditions we obtain infinitely many other triplets satisfying those conditions.	
	In \cite{3} the key point in finding a decomposition of $K_{r, r, s}$ into $5$-cycles, is writing $(r,r,s)$ satisfying the necessary conditions as a linear combination of two decomposed graphs $K_{1,3,3}$ and $K_{4,2,2}$.
	\begin{equation}\label{eq 1}
	\begin{bmatrix}
	r  \\
	s\\
	t
	\end{bmatrix}
	= 
	\begin{bmatrix}
	3 & 2  \\
	3 & 2  \\
	1 & 4
	\end{bmatrix}
	\begin{bmatrix}
	m \\
	n 
	\end{bmatrix}
	\end{equation}
	In other words, for all natural numbers $m$ and $n$, $K_{r, s, t}$ obtained from Equation (\ref{eq 1}) satisfies the necessary conditions. Since $(1,3,3)$, $(2,2,4)$, and $(5,5,5)$ are decomposable triplets, one may ask if any arbitrary linear combination of decomposable graphs yields a triplet $(r,s,t)$ satisfying the necessary conditions. The answer in general is negative but in some cases it works.
	\begin{lemma}
		The graph $K_{r,s,t}$ satisfies the necessary conditions if $(r, s, t)$ are obtained from the following equations for arbitrary $a,b,c \in \mathbb{N}$.
		\begin{center}
		\begin{equation*}		
			\begin{bmatrix}
			r  \\ s\\ t
			\end{bmatrix}
			=
			\begin{bmatrix}
			3 & 2 & 5 \\3 & 2 & 5 \\1 & 4 & 5
			\end{bmatrix}
			\begin{bmatrix}
			a \\b \\c
			\end{bmatrix}
			 ,\,\,\,
			\begin{bmatrix}
			r  \\s\\t
			\end{bmatrix}
			=
			\begin{bmatrix}
			3 & 2 & 17 \\	3 & 2 & 7 \\	1 & 4 & 19
			\end{bmatrix}
			\begin{bmatrix}
			a \\b \\c
			\end{bmatrix}
		\end{equation*}
		\end{center}
		
		\begin{center}
		\begin{equation*}
			\begin{bmatrix}
			r  \\s\\t
			\end{bmatrix}
			=
			\begin{bmatrix}
			5 & 10 & 20 \\3 & 4 & 30 \\	5 & 10 & 40
			\end{bmatrix}
			\begin{bmatrix}
			a \\b \\c
			\end{bmatrix}
			,\,\,\,
			\begin{bmatrix}
			r  \\s\\t
			\end{bmatrix}
			=
			\begin{bmatrix}
			5 & 10 & 15 \\3 & 4 & 11 \\	5 & 10 & 25
			\end{bmatrix}
			\begin{bmatrix}
			a \\b \\c
			\end{bmatrix}
		\end{equation*}	
	\end{center}
	\end{lemma}

	It is noted that the determinant of all matrices in the above Lemma are multiples of 100, but the Lemma is not correct for any matrix $D$ where $100 \ | \det(D).$\\
	The following Theorem shows that most acceptable triplets $(r,s,t)$ are obtained from $(1,3,3)$ and $(2,2,4)$, together with different multiples of 10 added to each part. 
	\begin{theorem}
		Let the graph $K_{r,s,t}$ satisfy the necessary conditions. 
		Then there exist positive integers $m,n, p_{1}, p_{2},$ and $p_{3}$ such that,
		$$r=m+4n+10p_{1}$$
		$$s=3m+2n+10p_{2}$$
		$$t=3m+2n+10p_{3}.$$
	\end{theorem}
	\begin{proof}
		By Corollary \ref{mod10}, we deduce $\{r,s,t\}$ is equal to one of the followings $\Mod{10}$.
		\begin{itemize}
			\item $r,s, \text{ and } t \text{ are odd and coprime with 5: } \{1,1,7\}, \{1,3,3\}, \{3,9,9\}, \{7,7,9\}$ 
			\item $ r,s, \text{ and } t \text{ are odd and two of them divide 5: } \{1,5,5\}, \{3,5,5\}, \{5,5,5\}, \{7,5,5\}, \{9,5,5\}$ 
			\item $ r,s, \text{ and } t \text{ are even and coprime with 5: } \{2,2,4\}, \{2,6,6\}, \{6,8,8\}, \{4,4,8\}$ 
			\item $ r,s, \text{ and } t \text{ are even and two of them divide 5: } \{0,0,0\}, \{0,0,2\}, \{0,0,4\}, \{0,0,6\}, \{0,0,8\}$ 
		\end{itemize}
		The assertion follows easily for $\{1,3,3\}, \{3,9,9\},$ and $ \{7,7,9\}$. Also, note that if $\{r,s,t\}\equiv \{1,1,7\}\,\, \Mod{10}$, then it is also equal to  $\{11,11,7\}$, where we get $m=3$ and $n=1$.\\
		For $\{2,2,4\}, \{2,6,6\}, \{6,8,8\},$ and $ \{4,4,8\}$ the assertion is immediate.\\
	\end{proof}
	Furthermore, the procedure of combining small decomposable graphs to obtain larger ones is made easier through the following Theorem.
	\begin{theorem}[Mixing Theorem] \label{thm 2.6}
		Let $(r_{i},s_{i},t_{i})$ be triplets for $i=1,\ldots,n$. Assume there exists a Latin square of size $n$ with label $k$ in $(i,j)^{th}$ entry. If $K_{r_{i},s_{j},t_{k}}$ can be decomposed into $5$-cycles for all $1\leq i,j \leq n$, then $K_{r_{1}+\cdots+r_{n}, \ s_{1}+\cdots+s_{n}, \ t_{1}+\cdots+t_{n}}$ has a $5$-cycle decomposition.
	\end{theorem}
	\begin{proof}
		The cell $(i,j)$ of the mentioned Latin square corresponds to the subgraph $K_{r_{i},s_{j},t_{k}}$ of the main graph.
		Therefore, edges of $K_{r_{1}+\cdots +r_{n}, s_{1}+\cdots+s_{n}, t_{1}+\cdots+t_{n}}$ decompose into the edges of $K_{r_{i},s_{j},t_{k}}$, when $1\leq i,j \leq n$ and $k$ is the $(i,j)^{th}$ entry of the Latin square. Since all $K_{r_{i},s_{j},t_{k}}$ are decomposable, one gets a decomposition of $K_{r_{1}+\cdots +r_{n}, s_{1}+\cdots+s_{n}, t_{1}+\cdots+t_{n}}$ by combining 5-cycles.
	\end{proof}
	\begin{example}
		If $K_{r_{1},s_{1},t_{1}}$, $K_{r_{1},s_{2},t_{2}}$, $K_{r_{2},s_{1},t_{2}}$, and $K_{r_{2},s_{2},t_{1}}$ are decomposable into $5$-cycles, then $K_{r_{1}+r_{2}, s_{1}+s_{2}, t_{1}+t_{2}}$ has a $5$-cycle decomposition, since the labels are extracted from a $2 \times 2$ Latin Square (Figure \ref{fig:6}).\\ 
		\begin{figure}[H]
		\centering
			\begin{subfigure}{0.31\textwidth}	
				\includegraphics[width=\linewidth]{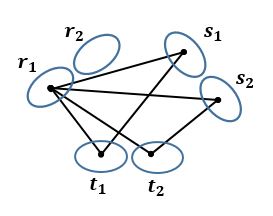}
				\caption{The subgraphs $K_{r_{1},s_{1},t_{1}}$ and $K_{r_{1},s_{2},t_{2}}$} \label{fig:6a}
			\end{subfigure}
			\begin{subfigure}{0.31\textwidth}
				\includegraphics[width=\linewidth]{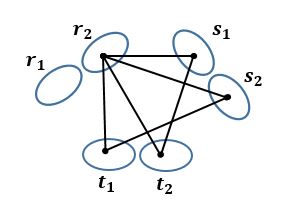}
				\caption{The subgraphs $K_{r_{2},s_{1},t_{2}}$ and $K_{r_{2},s_{2},t_{1}}$} \label{fig:6b}
			\end{subfigure}
			\begin{subfigure}{0.21\textwidth}
				\includegraphics[width=\linewidth]{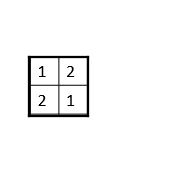}
				\caption{The Latin square of labels} \label{fig:6c}
			\end{subfigure}
			\caption{Decomposition of $K_{r_{1}+r_{2}, s_{1}+s_{2}, t_{1}+t_{2}}$ into copies of $K_{r_{1},s_{1},t_{1}}$, $K_{r_{1},s_{2},t_{2}}$, $K_{r_{2},s_{1},t_{2}}$, and $K_{r_{2},s_{2},t_{1}}$ using a $2 \times 2$ Latin square.} \label{fig:6}
		\end{figure}
	\end{example}
	%
	\section{Graph Adjacency Matrix (GAM) method}
	In order to decompose complete tripartite graphs into 5-cycles using the method in \cite{12,4,6,7,8} it is required to find an appropriate labeling of a Latin square and then try to cover it with irregular trades, specific subsets of edges and triangles that can be decomposed into $5$-cycles, 
	considering their labels and positions. When the graph parts are even, this method has a lengthy proof. Furthermore, Tackling the graphs with odd parts using this method is a complicated problem. 
	
	In this section, we use the graph adjacency matrix to decompose tripartite graphs into 5-cycles. Entries of the adjacency matrix represent edges of the graph so, decomposing the edges of a graph is equivalent to partitioning the adjacency matrix to specific patterns. We only deal with cell positions rather than handling labels and positions at the same time. This way, we transform the problem of cycle decomposition into covering sub-matrices with polygons.
	
	Let $D$ be the adjacency matrix of the complete tripartite graph $K_{r,s,t}$. Then $D$ has three zero-blocks of size $r$, $s$, and $t$ on its diagonal and all the other entries of $D$ are one as shown in Figure \ref{fig 7}. Moreover, the adjacency matrix $D$ is symmetric; each edge (e.g. $r_{i},s_{j}$) corresponds to two nonzero entries (e.g. $D(i, r+j)$ and $D(r+j, i)$). So we only consider the nonzero entries above the main diagonal as a representation of the whole graph.
		\begin{figure}[H]
			\centering
			\includegraphics[width=5.5cm]{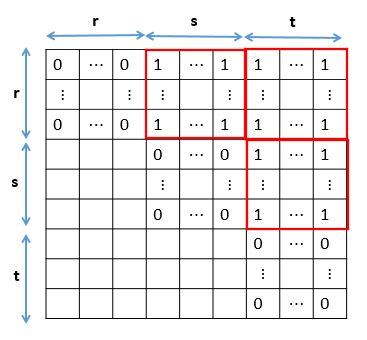}	
			\captionof{figure}{The adjacancy matrix $D$ of the graph $K_{r,s,t}$ with three blocks of zero on its diagonal and with other entries equal to one}\label{fig 7}
		\end{figure}
	\begin{figure}[H]
	\centering
		\begin{subfigure}{0.3\textwidth}
			\includegraphics[width=\linewidth]{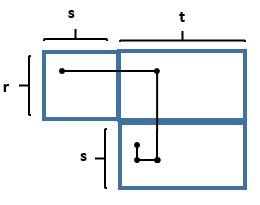}
			\caption{type 1} \label{fig:9a}
		\end{subfigure}
		\begin{subfigure}{0.3\textwidth}
			\includegraphics[width=\linewidth]{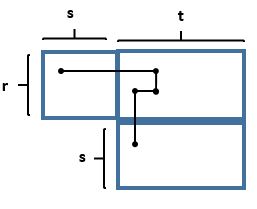}
			\caption{type 2} \label{fig:9b}
		\end{subfigure}
		\begin{subfigure}{0.3\textwidth}
			\includegraphics[width=\linewidth]{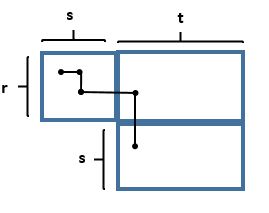}
			\caption{ type 3} \label{fig:9c}
		\end{subfigure}
		\caption{Representation of three types of 5-cycles in GAM method. Note that rows and columns of each rectangle are indexed independently starting from 1.} \label{fig:9}
	\end{figure}
	In the upper half of the adjacency matrix, a $5$-cycle corresponds to a path of length four (consisting of five end-points) and three right angles where the beginning column is the same as the final row (Figure \ref{fig:9}). In other words, the marked edges (shown with arrows) are identified as in Figure \ref{identify} and we have a pentagon with five right angles. This means we are looking for five edges connected in a cycle. Also, the pentagon must have three vertices in one part and one vertex in each of the remaining parts. Considering the part that vertices belong to, there exist three types of 5-cycles. In type 1, 2, and 3 of a 5-cycle, three vertices are located in the $s-t$, $t-r$, and $r-s$ sub-matrix, respectively.
	 We call the procedure of decomposing $K_{r,s,t}$ into 5-cycles by finding such pentagons in the adjacency matrix, as the \textbf{Graph Adjacency Matrix} (\textbf{GAM}) method.
	 By trying to cover the upper half of the adjacency matrix with polygons, we can decompose graphs into 5-cycles. We denote 5-cycles with different colors/ labels. When the graph adjacency matrix is colored thoroughly, all edges are part of a 5-cycle and the decomposition is complete.
	 %
		\begin{figure}[H]
			\centering
			\includegraphics[width=4.5cm]{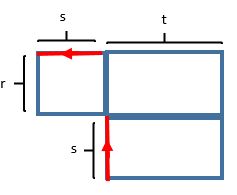}
			\includegraphics[width=4.5cm]{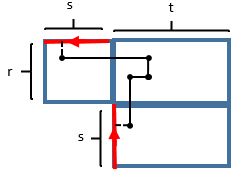}
			\captionof{figure}{The adjacency matrix of the graph $K_{r,s,t}$ and a 5-cycle in GAM method}
			\label{identify}
		\end{figure}
	\begin{example}
		Figure \ref{fig:5} shows decomposition of $K_{1,3,3}$ and $K_{2,2,4}$ using GAM method.
		\begin{figure}[H]
		\centering
			\begin{subfigure}{0.27\textwidth}
				\includegraphics[width=\linewidth]{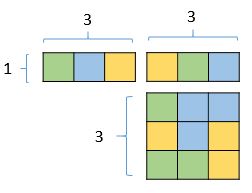}
				\caption{ $K_{1,3,3}$} \label{fig:5a}
			\end{subfigure}
			\begin{subfigure}{0.27\textwidth}
				\includegraphics[width=\linewidth]{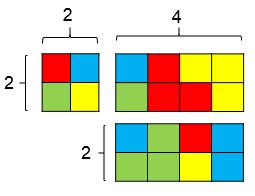}
				\caption{ $K_{2,2,4}$} \label{fig:5b}
			\end{subfigure}
			\caption{Decomposition of $K_{1,3,3}$ and $K_{2,2,4}$ into 5-cycles using GAM method. 5-cycles are denoted by cells of the same color.}
			\label{fig:5} 
		\end{figure}
	\end{example}
\begin{example}
	Figure \ref{fig matlab} uses GAM method to show how $K_{r,r,s}$ is decomposed into $5$-cycles as in \cite{3}. Note that in this case the number of cycles are $c_{1}=c_{2}=rn$, and $c_{3}=rm$ where $m$ and $n$ are calculated as in Equation \ref{eq 1}.
	\begin{figure}[H]
		\centering
		\includegraphics[width=5cm]{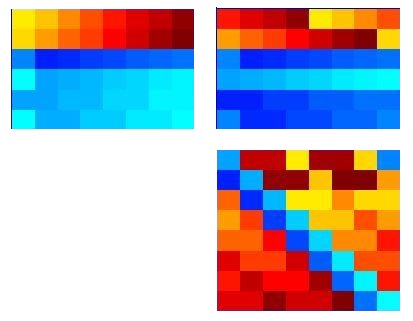} 
		\captionof{figure}{GAM representation of decomposition of $K_{8,8,6}$ using cycles of \cite{3}.} 
		\label{fig matlab}
	\end{figure}
\end{example}
	In subsequent, we will use GAM method to decompose more graphs into 5-cycles.
	\begin{theorem}
		Let $G=K_{r, s, t}$ satisfy the necessary conditions and $s=2r$. Also, assume $G$ has no cycle of type 3. Then $(r, s, t)$ is a multiple of $(6,12,16)$ and the $5$-cycles may be found explicitly.
	\end{theorem}
	\begin{proof}
		Using Proposition \ref{prop 2.2}, the number of cycles of type 3 is $\frac{(-st-rt+4rs)}{10}$ and $s=2r$ so we have 
		$$\frac{(-2rt-rt+4r(2r))}{10}=0 \ \  \Longrightarrow   \ \ -3rt+8r^{2}=0 \,\,\, \Longrightarrow \,\,\, t=\frac{8}{3}r. $$ Now, since 
		$(r,s,t)=(r,2r,\frac{8}{3}r)$, $r$ must be an even multiple of 3 and hence $(r,s,t)=(6k,12k,16k)$ for some $k \in \mathbb{N}$.
		In \ref{appendix A},
		one sees the 5-cycle decomposition of $(6,12,16)$ in GAM representation. It is immediate to find decomposition of all $(6k,12k,16k)$ using the same strategy.
	\end{proof}

	\begin{proposition}
		$K_{10k,12k,20k}$ has a decomposition into $5$-cycles.
	\end{proposition}
	\begin{proof}
		This graph is decomposed using GAM method. The 5-cycles of $K_{10,12,20}$ are illustrated in \ref{appendix B} and hence all $K_{10k,12k,20k}$ are decomposable. 
	\end{proof}
	\section{Decomposition of graphs in more cases}
	Hereafter, we will use GAM method to decompose some graphs with odd parts. For example,  $K_{9,13,19}$ and $K_{11,13,23}$ are the smallest graphs for which no decomposition were known before. We find decompositions for these and some more graphs with odd parts.
	\begin{proposition} \label{9-13-19}
		$\bm{K_{9,13,19}}$ and $\bm{K_{11,13,23}}$ and their multiples have decompositions into 5-cycles.
	\end{proposition}
	\begin{proof}
		These graphs are decomposed using GAM method. The pictures in \ref{appendix C} and \ref{appendix D} show these decompositions in several steps.\\
	\end{proof}
	\begin{theorem}
		Let $K_{r,s,t}$ with $r \leq s \leq t$ be decomposable into $5$-cycles and $p$ be a multiple of $5$ with $p \equiv r \Mod{2}$
		and $t$ satisfying $\frac{t}{2} \leq p \leq 3r$. Then $K_{r+2p,s+2p,t+2p}$ also has a decomposition into $5$-cycles.
	\end{theorem}
	\begin{proof}
		Here $(r,p,p)$ satisfies the necessary conditions  whenever
		$$p \leq \dfrac{4rp}{r+p}  \ and  \ r \leq \dfrac{4p^{2}}{2p} \ \ \Longleftrightarrow \ \  \frac{r}{2} \leq p \leq 3r $$
		Hence, by replacing $r,s,$ and $t$ in the above argument one sees that $k_{r,p,p}$, $k_{s,p,p}$, $k_{t,p,p}$ satisfy the necessary conditions  whenever
		$$\frac{r}{2} \leq p \leq 3r , \ \ \frac{s}{2} \leq p \leq 3s, \ \ \frac{t}{2} \leq p \leq 3t $$
		which is satisfied by the hypothesis. Consequently, they have decompositions into $5$-cycles \cite{4}.
		Consider the following combination:
		\begin{equation}
		\begin{aligned}
		(r, s, t)+ (p,p,p)+ (p,p,p) =(r+2p,\ s+2p,\ t+2p)
		\end{aligned}
		\end{equation}
		As $k_{r,s,t},\ k_{r,p,p}, \ k_{s,p,p}, \ k_{t,p,p}, \ $ and $k_{p,p,p}$ all have decompositions into $5$-cycles, we are able to use Mixing Theorem \ref{thm 2.6} and conclude that $K_{r+2p,s+2p,t+2p}$ has also a decomposition into $5$-cycles (Figure \ref{fig:12}).
		\begin{figure}[H]
			\begin{floatrow}
				\ffigbox{%
					\includegraphics[width=0.6\linewidth]{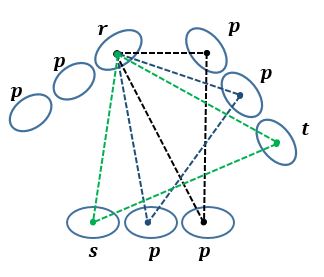}%
				}{%
					\caption{$K_{r+2p,s+2p,t+2p}$ is decomposed into copies of decomposable graphs by subgraphs obtained from Table \ref{table 2}. Then using Mixing Theorem \ref{thm 2.6} we get a 5-cycle decomposition of $K_{r+2p,s+2p,t+2p}$.}\label{fig:12}
				}
				\capbtabbox{%
					\begin{tabular}{c|c c c} 
						\textbf{ } & \textbf{s} & \textbf{p} & \textbf{p}\\
						\hline
						\textbf{r} & t & p & p\\
						\textbf{p} & p & t & p\\
						\textbf{p} & p & p & t\\
					\end{tabular}
				}{%
					\caption{Latin square of the Mixing Theorem \ref{thm 2.6}.}\label{table 2}
				}
			\end{floatrow}
		\end{figure}
		
	\end{proof}
	\begin{corollary}
		In the above circumstances, we can add as many components of $(p,p,p)$ as we want and prove that $K_{r+np,s+np,t+np}$ has a $5$-cycle decomposition for all $n \in \mathbb{N} $.
	\end{corollary}
	\begin{example}
		$\bm{K_{37,47,49}}$ has a $5$-cycle decomposition since $(37,47,49)=(7+30, 17+30,19+30)$, and $K_{7,17,19}$ has a decomposition. 
		Also, $\bm{K_{39,43,49}}$ has a $5$-cycle decomposition since $(39,43,49)=(9+30, 13+30,19+30)$, and $K_{9,13,19}$ has a decomposition as shown in Proposition \ref{9-13-19}.
	\end{example}
	
	\begin{theorem}
		If $K_{a,b,c}$, $K_{5k,b,c}$, $K_{b,5k,5k}$, $K_{a,5k,5k}$, and $K_{c,5k,5k}$ are decomposed into $5$-cycles, then $K_{a+10k,2b+5k,2c+5k}$ can also be decomposed.
	\end{theorem}
	\begin{proof}
		We have:
		\begin{equation}
			(a,b,c) + (5k, b, c) + (5k, 5k, 5k) = (a+10k,\ 2b+5k,\ 2c+5k)
		\end{equation}
		We divide $K_{a+10k,2b+5k,2c+5k}$ into smaller decomposable graphs using Table 
		\ref{table 3}. Then, by Mixing Theorem \ref{thm 2.6} the assertion follows.
			\begin{table}[h!]
			\begin{center}
				\caption{The Latin square of Mixing Theorem \ref{thm 2.6} for $K_{a+10k,2b+5k,2c+5k}$.}\label{table 3}
				\begin{tabular}{c|c c c} 
					\textbf{ } & \textbf{b} & \textbf{b} & \textbf{5k}\\
					\hline
					\textbf{c} & a & 5k & 5k\\
					\textbf{c} & 5k & a & 5k\\
					\textbf{5k} & 5k & 5k & a\\
				\end{tabular}
			\end{center}
		\end{table} 
	\end{proof}
	\begin{example}
		$\bm{K_{49,45,55}}$ admits a decomposition since $(45,49,55)=(15+30,2 \times 17+15,2 \times 25+15 )$ and $K_{15,17,25}$ is decomposed.
	\end{example}
	\begin{theorem}\label{5k}
		If $K_{a,b,5k}$, $K_{a,c,5k}$, $K_{a,5k,5k}$, $K_{b,5k,5k}$, and $K_{c,5k,5k}$ are decomposable into $5$-cycles, then $K_{a+10k,b+c+5k,15k}$ may also be decomposed.
	\end{theorem}
	\begin{proof}
		We have:
		\begin{equation}
		(a,b,5k) + (5k, c, 5k) + (5k, 5k, 5k) = (a+10k,\ b+c+5k,\ 15k)
		\end{equation}
		We combine complete tripartite graphs using Table \ref{table 33}, and by Mixing Theorem \ref{thm 2.6} the assertion follows.
		\begin{table}[h!]
			\begin{center}
				\caption{The Latin square of Mixing Theorem \ref{thm 2.6} for $K_{a+10k,b+c+5k,15k}$.}\label{table 33}
				\begin{tabular}{c|c c c} 
					\textbf{ } & \textbf{b} & \textbf{c} & \textbf{5k}\\
					\hline
					\textbf{a} & 5k & 5k & 5k\\
					\textbf{5k} & 5k & 5k & 5k\\
					\textbf{5k} & 5k & 5k & 5k\\
				\end{tabular}
			\end{center}
		\end{table} \label{table 4}
	\end{proof}
	\begin{example}
		$\bm{K_{55,47,45}}$ admits a decomposition into 5-cycles. Put $a=25, \ b=13, \ c=19 $ and use Theorem \ref{5k}. Also,
		$\bm{K_{55,51,45}}$ admits a decomposition into 5-cycles. Use Theorem \ref{5k} for $a=25$, $b=17$, $c=19$.
	\end{example}
	In the following, we complete a result from \cite{7}, where the case with two parts multiples of 5 and $r \leq 75$ remained unsolved.
	\begin{theorem}\label{thm 4.9}
		Let $K_{r,s,t}$ satisfy the necessary conditions where $r,s,t$ are odd $r \in \{31,33,\ldots,75 \}$ and $s,t \in \{45,55,65,75 \}$. Then $K_{r,s,t}$ has a $5$-cycle decomposition.
	\end{theorem}
	\begin{proof}
		It is possible to write $r=r_{1}+r_{2}+r_{3}$ with $r_{i} \in \{11,13,15,17,19,25 \}$. Since all $K_{r_{i},15,15}$, $K_{r_{i},15,25}$, and $K_{r_{i},25,25}$ that appear are decomposable \cite{4,7} and $K_{r,s,t}$ is a \emph{summation} of three such graphs, by Mixing Theorem \ref{thm 2.6}, $K_{r,s,t}$ is also decomposable.
	\end{proof}
	\begin{theorem}\label{thm 4.10}
		Let $K_{r,s,t}$ satisfy the necessary conditions where $r,s,t$ are odd $r \in \{55,57,\ldots, 125\}$ and $s,t \in \{75,85,95,105,115,125 \}$. Then $K_{r,s,t}$ has a $5$-cycle decomposition.
	\end{theorem}
	\begin{proof}
		It is possible to write $r=r_{1}+\cdots+r_{5}$ with $r_{i} \in \{11,13,15,17,19,25 \}$. Since $K_{r_{i},15,15}$, $K_{r_{i},15,25}$, and $K_{r_{i},25,25}$ are decomposable and $K_{r,s,t}$ is a \emph{summation} of five such graphs, by Mixing Theorem \ref{thm 2.6}, $K_{r,s,t}$ is also decomposable.
	\end{proof}
	
	As we see, Theorems \ref{thm 4.9} and \ref{thm 4.10} can be generalized to combine several graphs to obtain more decompositions.
	\begin{theorem}
		Let $m$ be an odd integer and suppose $K_{r,s,t}$  with $11m \leq r \leq 25m $ and  $15m \leq s,t \leq 25m$ satisfy the necessary conditions where $r,s,t$ are odd and at least two parts are multiples of $5$. Then $K_{r,s,t}$ has a $5$-cycle decomposition.
	\end{theorem}
	\begin{proof}
		It is possible to write $t=t_{1}+\cdots+t_{m}$ and $s=s_{1}+\cdots+s_{m}$ with $s_{i}, t_{i} \in  \{15,25\}$. Also, since $r \leq 25m$ it is possible to write $r=r_{1}+\cdots+r_{m}$ with $r_{i} \in \{11,13,15,17,19,25 \}$.
		Moreover, since $K_{r_{i},15,15}$, $K_{r_{i},15,25}$, and $K_{r_{i},25,25}$ are decomposable and $K_{r,s,t}$ is a summation of these graphs, by Mixing Theorem \ref{thm 2.6}, $K_{r,s,t}$ is also decomposable.\\
	\end{proof}
	Note that if $t=3r$, then clearly $(r,s,t)$ is a multiple of $(1,3,3)$ and hence $k_{r,3r,3r}$ is decomposable. So, $K_{r,s,t}$ where $s,t$ are multiples of $5$ is proved to admit a $5$-cycle decomposition unless when $2.27r \le t < 3r$. The above theorems solve many of the cases in this gap.
	
	Finally, we would like to introduce the \textbf{strategy} we used for decomposing some graphs (e. g. $K_{11,13,23}$) with the hope to facilitate decomposition of the remaining graphs satisfying the necessary conditions. 
	\\Note that cycles of type 1 and 3 always differ by a multiple of $s$.
	\begin{lemma} \label{lemma s}
		Let $c_{1}$, $c_{2}$, $c_{3}$ be the number of cycles of type 1, 2, and 3 in $K_{r,s,t}$, respectively. Then we have $c_{1}\equiv c_{3} \ \Mod{s}$.
	\end{lemma}
	\begin{proof}
		We show  $c_{1}-c_{3}\equiv 0 \ \Mod{s}$. Since $r$ and $t$ are odd integers, $\frac{t-r}{2}$ is an integer and we have\\
		$ \frac{4st-rt-rs}{10}-\frac{-st-rt+4rs}{10} \equiv \frac{5st-5rs}{10} \equiv \frac{st-rs}{2} \equiv s\frac{t-r}{2} \equiv 0  \ \Mod{s}$
	\end{proof}
	
	\textbf{Strategy}\\
	Let $K_{r,s,t}$ with $r \leq s \leq t$ be a complete tripartite graph satisfying the necessary conditions.
	\begin{itemize}
		\item [1.] Use the GAM method and draw the upper part of the adjacency matrix for $K_{s,r,t}$.
		\item [2.] Find cycles of type 1 with the following formula and let at least 4 columns remain. Note that in this step, one should either fill an entire column or leave it empty. If necessary, in one step we may shift all vertices labeled $t_{i}$ by a constant (e.g. 1 or 2).
		$$(r_{i},s_{j},t_{j+3(i-1)},s_{j+2},t_{j+3(i-1)+1})$$
		$$ \text{for} \ j=1, \ldots ,s \ \text{and} \ i=1,\ldots, \min \{r-4, \lfloor c_{1}/s \rfloor\} $$
		\item [3.] Find cycles of type 2 using one column of the first part together with a diagonal in the second part, till only 4 columns remain unoccupied.
		\item [4.] In the remaining 4 columns use the pattern as in Figure \ref{fig:10} (pink parts) to find $c_{1}'\equiv c_{1} \  \Mod{s}$ cycles of type 1 and $c_{3}'\equiv c_{3} \ \Mod{s}$ cycles of type 3. This is possible as proved in Lemma \ref{lemma s}. 
		\item [5.] At this point, one should try to fill the adjacency matrix in such a way that the part beneath occupied entries are filled first. This step works with try and error so we called the whole process a strategy rather than an algorithm.
	\end{itemize}
	The use of this Strategy is illustrated in decomposition of $K_{11,13,23}$ (See Figure \ref{fig:10}).

	\section{Conclusions}
	
	When a complete tripartite graph $k_{r,s,t}$ has a decomposition into 5-cycles, then the number of edges must divide 5; $(r,s,t)$ must be all even or all odd; and $t \leq 4rs/(r+s)$. These necessary conditions are proved to be sufficient in some cases, however their sufficiency in general is a conjecture. 
	
	In this paper, the Graph Adjacency Matrix (GAM) method was introduced as an approach for 5-cycle decomposition. This method is an appropriate tool to explain the logic behind previous graph decompositions. Consequently, using this method we decomposed $K_{9,13,19}$ and $K_{11,13,23}$ for the first time. Moreover, providing a “combining rule” for decomposable graphs, stated as Mixing Theorem \ref{thm 2.6}, we decomposed more graphs into $5$-cycles for some of which no decomposition was known before. The following results are deduced from the Mixing Theorem:
		\begin{itemize}
		\item $K_{r,s,t}$  with $11m \leq r \leq 25m $ and  $15m \leq s,t \leq 25m$ was shown to have a decomposition for $r,s,t$ and $m$ odd where at least two parts are multiples of $5$.
		\item $K_{r+np,s+np,t+np}$ has a 5-cycle decomposition provided that $K_{r,s,t}$ has a 5-cycle decomposition where $p$ is a multiple of $5$ with $p \equiv r \Mod{2}$ and $\frac{t}{2} \leq p \leq 3r$ for $r \leq s \leq t$. 
		\item $K_{a+10k,2b+5k,2c+5k}$ is decomposed into 5-cycles, if $K_{a,b,c}$, $K_{5k,b,c}$, $K_{b,5k,5k}$, and $K_{c,5k,5k}$ are decomposable.
		\item $K_{a+10k,2b+5k,2c+5k}$ is decomposed into 5-cycles, if $K_{a,b,c}$, $K_{5k,b,c}$, $K_{b,5k,5k}$, and $K_{c,5k,5k}$ are decomposable.
		\end{itemize}
		
	Moreover, some new relations and properties of the graphs satisfying the necessary conditions were found with the hope to illuminate the unknown aspects of this Conjecture. The Graph Adjacency Matrix (GAM) method is a general approach that can be used in many graph decomposition problems. 
	Furthermore, if the correctness of Conjecture could be proved for the remaining cases, then one could say the necessary conditions are also sufficient.
	A historical review of the problem of decomposing complete tripartite graphs into 5-cycles is provided in Table \ref{tab:table1}.
	
	\subsection{List of primitive graphs}\label{appendix E}
	Here, we present the list of primitive graphs $k_{r, s, t}$, where $\gcd (r,s,t) = 1,2,$ or $5$,
	with $r\leq 11$ that satisfy the necessary conditions as well as their decomposition status. Moreover, we mark that the smallest undecomposed graph till now is $K_{9,19,23}$. 
	In this work, we proved decomposability of infinitely many graphs into 5-cycles for the first time (e. g. $K_{9,13,19}$, $K_{11,13,23}$,
	$K_{29,45,55}$, $K_{29,45,65}$, $K_{29,55,65}$, $K_{29,55,75}$, $K_{29,65,75}$, $K_{31,45,65}$). 
	\begin{table}[H]
		\begin{center}
			\caption{Some primitive graphs satisfying the necessary conditions and their decomposition status.}
			\label{tab:table1}
			\begin{tabular}{c|c|c|c} 
				\textbf{r} & \textbf{s} & \textbf{t}& \textbf{Status-year}\\
				\hline
				1 & 3 & 3 & decomposed \cite{3}-1995\\
				\hline
				2 & 2 & 4 & decomposed \cite{3}-1995\\
				2 & 6 & 6 &  decomposed \cite{3}-1995\\
				\hline
				3 & 5 & 5 &  decomposed \cite{3}-1995\\
				\hline
				4 & 10 & 10 &  decomposed \cite{3}-1995\\
				\hline
				5 & 5 & 5 &  decomposed \cite{3}-1995\\
				5 & 5 & 7 &  decomposed \cite{4}-2000\\
				5 & 5 & 9 &  decomposed \cite{3}-1995\\
				\hline
				6 & 8 & 8 &  decomposed \cite{3}-1995\\
				6 & 12 & 16 & decomposed \cite{12}-2000\\
				& & & also decomposed in this work\\
				\hline
				7 & 7 & 9 &  decomposed \cite{3}-1995\\
				7 & 11 & 11 & decomposed \cite{3}-1995\\
				7 & 15 & 15 &  decomposed \cite{3}-1995\\
				7 & 17 & 19 & decomposed \cite{7}-2012\\
				\hline
				8 & 10 & 10 & decomposed \cite{3}-1995\\
				8 & 14 & 14 & decomposed \cite{3}-1995\\
				8 & 16 & 18 & decomposed \cite{12}-2000\\
				\hline
				9 & 9 & 13 & decomposed \cite{3}-1995\\
				9 & 13 & 19 & decomposed in this work\\
				9 & 17 & 17 & decomposed \cite{3}-1995\\
				9 & 19 & 23 & \textit{not decomposed yet.}\\
				9 & 25 & 25 & decomposed \cite{3}-1995\\
				\hline
				10 & 10 & 12 &  decomposed \cite{3}-1995\\
				10 & $s$ & $t$ & (all even) decomposed \cite{12}-2000\\
				10 & 20 & 26 & decomposed \cite{12}-2000\\
				\hline
				11 & 11 & 17 & decomposed \cite{3}-1995\\
				11 & 13 & 13 & decomposed \cite{3}-1995\\
				11 & 13 & 23 & decomposed in this work\\
				$\cdots$ & $\cdots$ & $\cdots$ &\\
			\end{tabular}
		\end{center}
	\end{table} \label{table1}
	
	\bibliographystyle{plain} 
	\bibliography{biblio}
	
	\newpage
	\section{Appendix}
	\subsection{Decomposition of $K_{6,12,16}$}\label{appendix A}
	Here, we give the 5-cycle decomposition of $K_{6,12,16}$ in GAM representation, step by step. 
	When all cells of the graph adjacency matrix are colored, each edge is part of a 5-cycle and the decomposition is complete.
	\begin{center}
		\includegraphics[width=9cm]{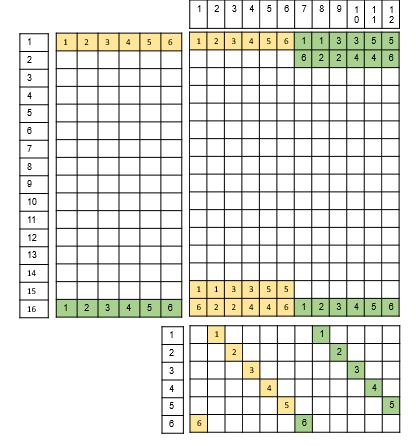}\\
		\includegraphics[width=9cm]{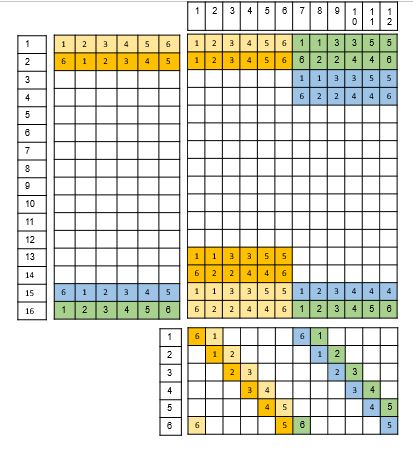}\\
		\includegraphics[width=9cm]{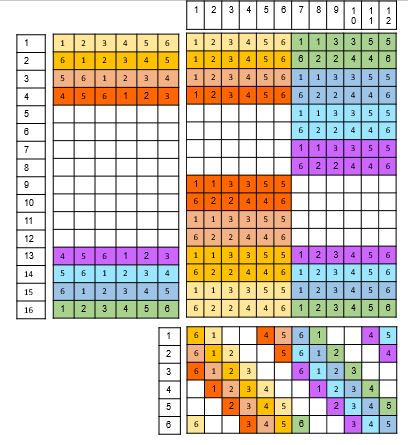}\\
		\includegraphics[width=9cm]{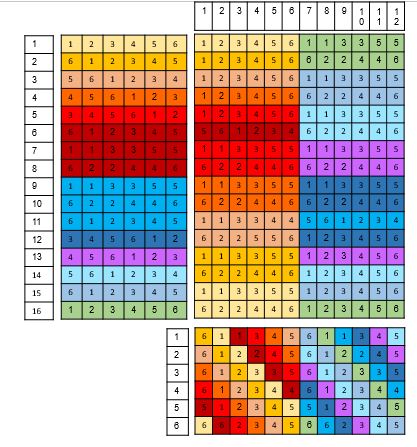}
		\captionof{figure}{Decomposition of $K_{6,12,16}$ into 5-cycles using GAM method}
	\end{center}
	\newpage
	\subsection{Decomposition of $K_{10,12,20}$}\label{appendix B}
	Here, we give the 5-cycle decomposition of $K_{10,12,20}$ in GAM representation, step by step.
	\begin{center}
		\includegraphics[width=12cm]{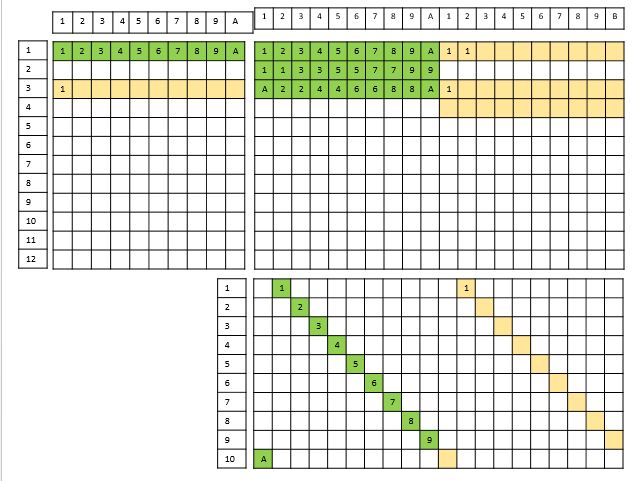}\\
		\includegraphics[width=12cm]{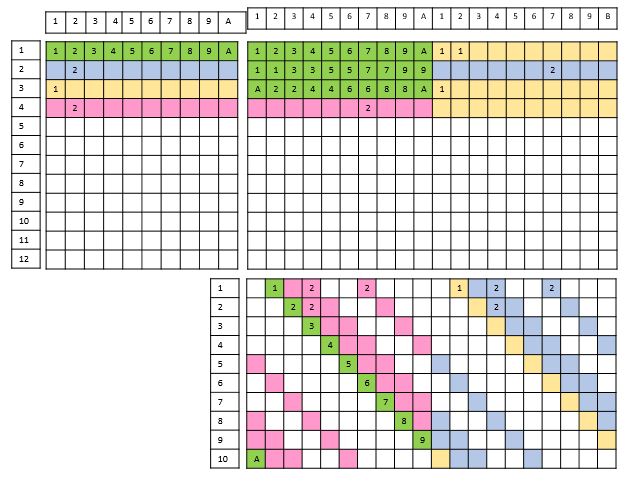}\\
		\includegraphics[width=12cm]{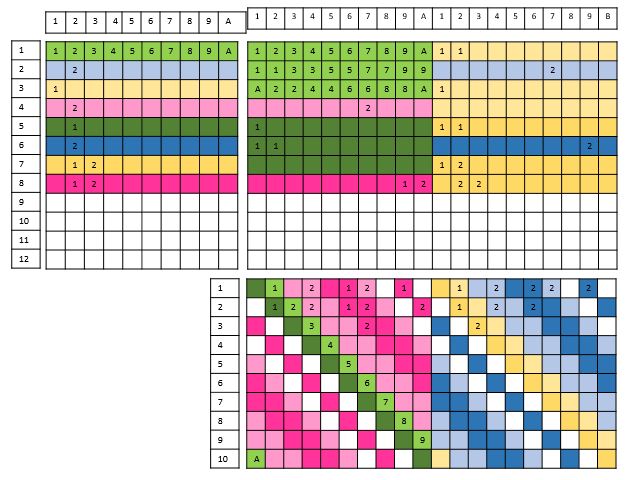}\\
		\includegraphics[width=12cm]{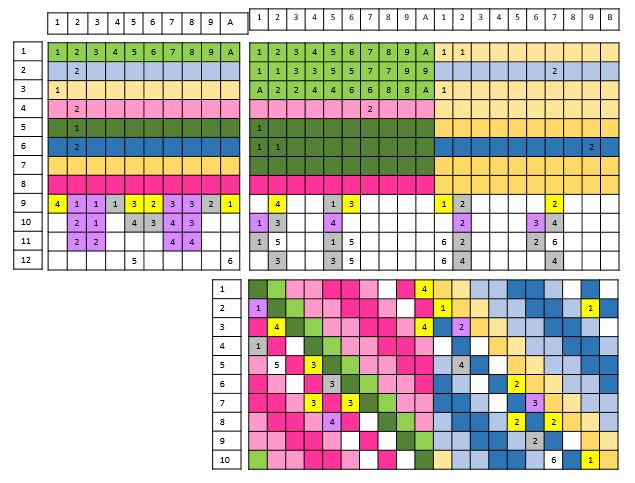}\\
		\includegraphics[width=12cm]{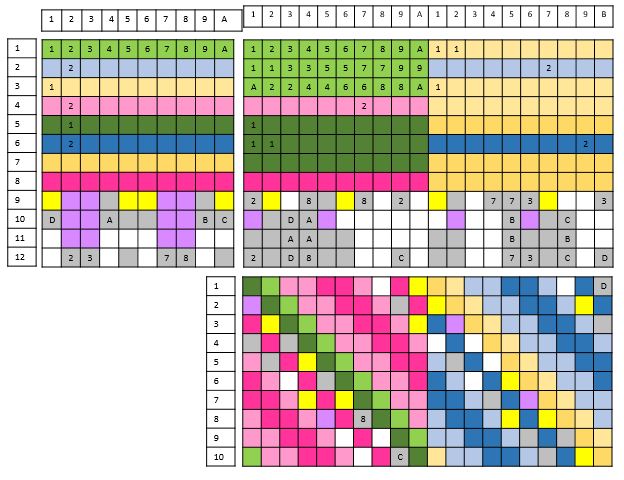}\\
		\includegraphics[width=12cm]{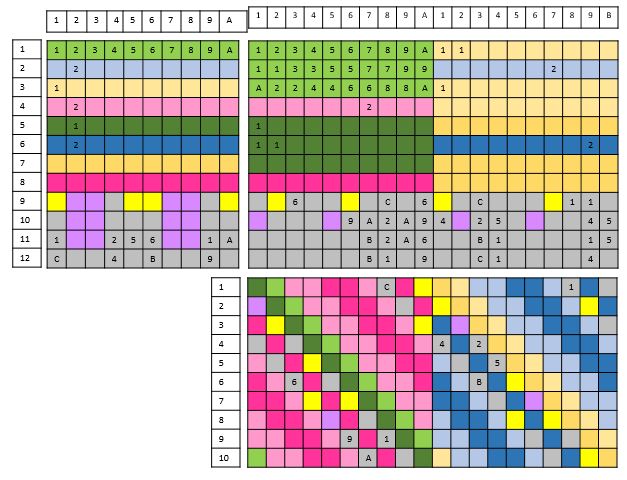}
		\captionof{figure}{Decomposition of $K_{10,12,20}$ into 5-cycles using GAM method}
	\end{center}
	\newpage
	\subsection{Decomposition of $K_{9,13,19}$}\label{appendix C}
	Here, we give the 5-cycle decomposition of $K_{9,13,19}$ in GAM representation, step by step. 
	\begin{center}
		\includegraphics[width=12cm]{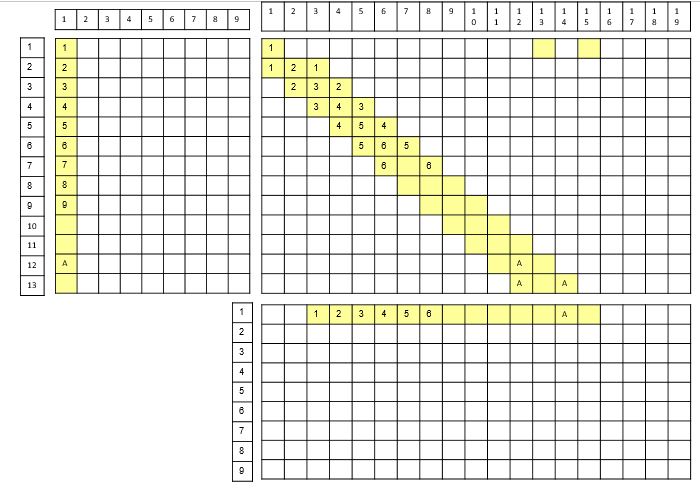}\\
		\includegraphics[width=12cm]{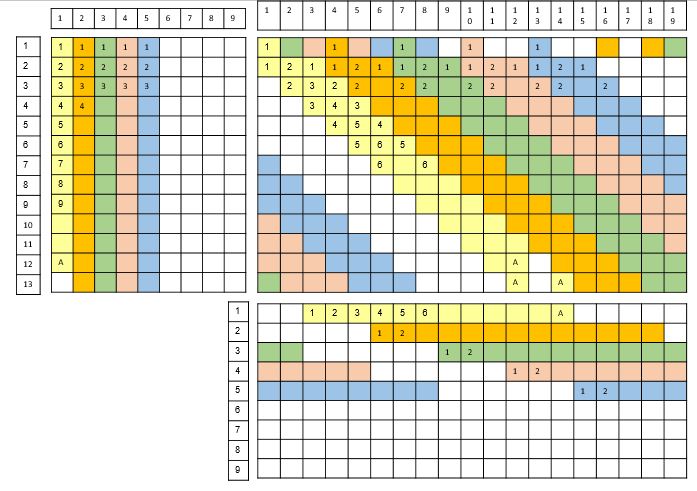}\\
		\includegraphics[width=12cm]{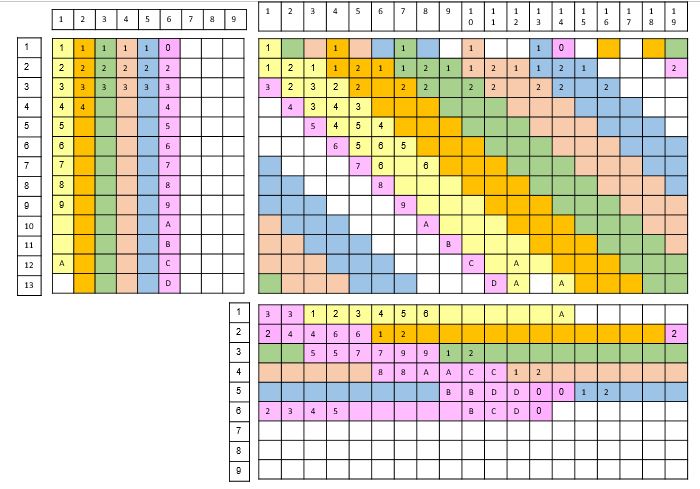}\\
		\includegraphics[width=12cm]{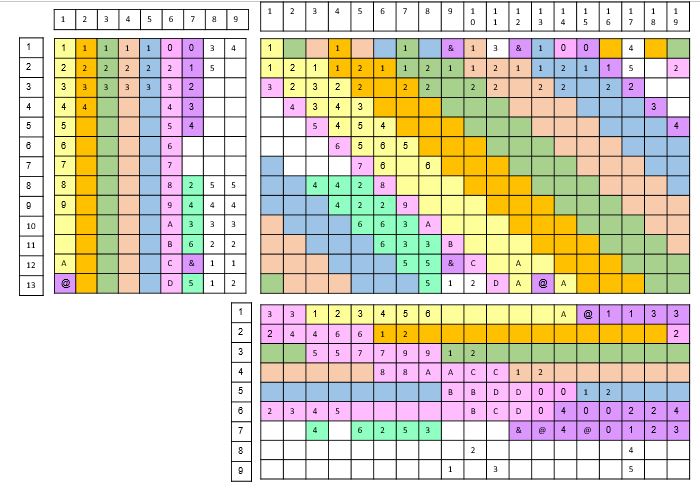}\\
		\includegraphics[width=12cm]{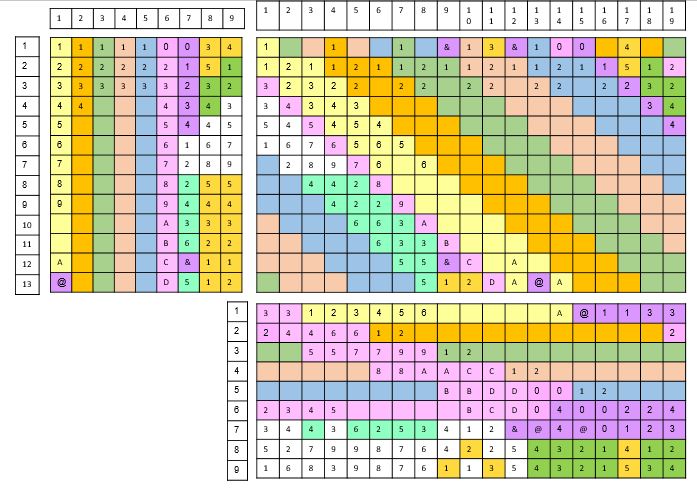}\\
		\captionof{figure}{Decomposition of $K_{9,13,19}$ into 5-cycles using GAM method}
	\end{center}
	\subsection{Decomposition of $K_{11,13,23}$}\label{appendix D}
	Here, we give the 5-cycle decomposition of $K_{11,13,23}$ in GAM representation, step by step. 
	\begin{center}
		\includegraphics[width=12cm]{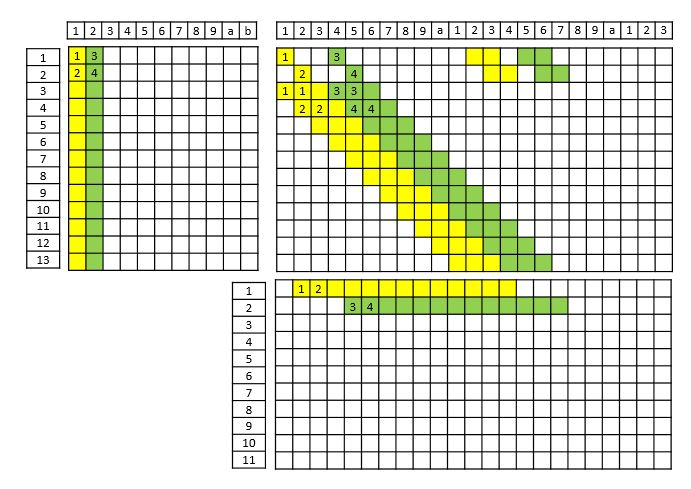}\\
		\includegraphics[width=12cm]{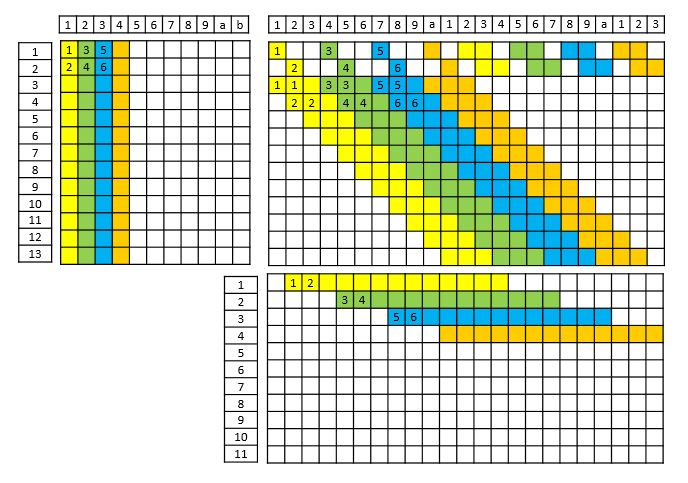}\\
		\includegraphics[width=12cm]{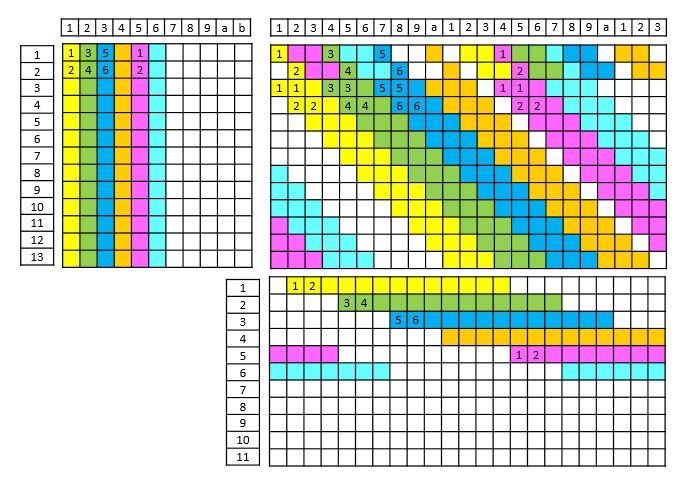}\\
		\includegraphics[width=12cm]{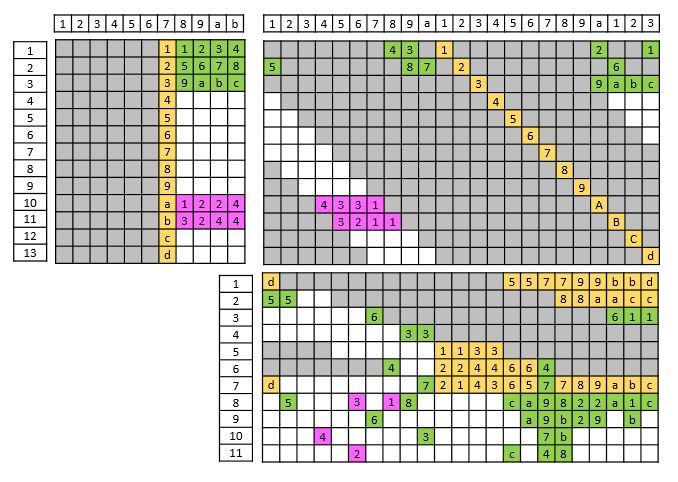}\\
		\includegraphics[width=12cm]{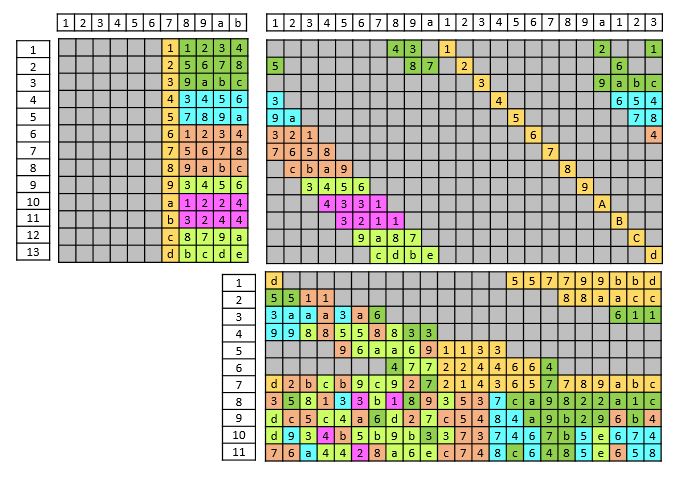}\\
		\captionof{figure}{Decomposition of $K_{11,13,23}$ into 5-cycles using GAM method} \label{fig:10}
	\end{center}
	
	
\end{document}